\documentclass[12pt,a4paper,reqno]{amsart}
\usepackage[a4paper,top=3cm,bottom=3cm,inner=3cm,outer=3cm,includehead,includefoot]{geometry}
\usepackage{amsmath,amssymb,amsthm,bbm,mathtools,wasysym,calc,verbatim,enumitem,tikz,pgfplots,url,hyperref,mathrsfs,cite,fullpage}
\usetikzlibrary{shapes.misc,calc,intersections,patterns,decorations.pathreplacing,backgrounds}
\hypersetup{colorlinks=true,citecolor=blue,linktoc=page}

\newtheorem{theorem}{Theorem}
\newtheorem{lemma}[theorem]{Lemma}
\newtheorem{prop}[theorem]{Proposition}
\newtheorem{claim}[theorem]{Claim}
\newtheorem{corollary}[theorem]{Corollary}

\theoremstyle{definition}
\newtheorem{definition}[theorem]{Definition}

\newtheorem*{HistRmk}{Historical Remark}

\newenvironment{clmproof}[1]{\begin{proof}[Proof of Claim~\ref{#1}]\let\qednow\qedsymbol\renewcommand{\qedsymbol}{}}{\; \qednow \end{proof}}

\numberwithin{theorem}{section}
\setlist[itemize]{leftmargin=1cm}
\setlist[enumerate]{leftmargin=1cm}

\everydisplay{\thickmuskip=7mu}
\renewcommand{\leq}{\leqslant}
\renewcommand{\geq}{\geqslant}
\renewcommand{\le}{\leqslant}
\renewcommand{\ge}{\geqslant}
\renewcommand{\to}{\rightarrow}
\newcommand{\interior}{\operatorname{int}}
\newcommand{\conv}{\operatorname{conv}}

\let\eps\varepsilon

\def\Ex{\mathbb{E}}

\def\F{\mathcal{F}}

\def\H{\mathbb{H}}
\def\HH{\mathcal{H}}

\def\N{\mathbb{N}}
\def\cN{\mathcal{N}}
\def\O{\mathcal{O}}
\def\P{\mathbb{P}}

\def\Pr{\mathbb{P}}

\def\QQ{\mathcal{Q}}
\def\R{\mathbb{R}}

\def\S{\mathcal{S}}
\def\SS{\mathbf{S}}
\def\T{\mathcal{T}}
\def\U{\mathcal{U}}

\def\X{\mathcal{X}}
\def\Y{\mathcal{Y}}
\def\Z{\mathbb{Z}}

\def\<{\langle}
\def\>{\rangle}
\def\0{\mathbf{0}}
\def\1{\mathbbm{1}}

\def\edge{\partial}
\def\str{\S^*}
\def\Vor{\mathrm{Vor}}
\def\Cell{\mathrm{Cell}}

\title{Subcritical monotone cellular automata}

\author[P. Balister \and B. Bollob\'as \and R. Morris \and P.J. Smith]{Paul Balister \and B\'ela Bollob\'as \and Robert Morris \and Paul Smith}

\address{Mathematical Institute, University of Oxford, Radcliffe Observatory Quarter, Woodstock Road, Oxford, OX2 6GG, UK}\email{Paul.Balister@maths.ox.ac.uk}

\address{Department of Pure Mathematics and Mathematical Statistics, Wilberforce Road, Cambridge, CB3 0WA, UK, and Department of Mathematical Sciences, University of Memphis, Memphis, TN 38152, USA}
\email{b.bollobas@dpmms.cam.ac.uk}

\address{IMPA, Estrada Dona Castorina 110, Jardim Bot\^anico, Rio de Janeiro, 22460-320, Brazil}
\email{rob@impa.br}

\address{Clerkenwell, London}
\email{paulsmith@cantab.net}
\thanks{P.B.\ and B.B.\ were partially supported by NSF grant DMS~1855745, 
R.M.\ by FAPERJ (Proc.~E-26/202.993/2017) and CNPq (Proc.~304237/2016-7), and by the ERC Starting Grant 680275 MALIG, and P.S. by Israel Science Foundation grant 1147/14 and by a CNPq bolsa PDJ} 

\begin{document}

\begin{abstract}
We study monotone cellular automata (also known as $\U$-bootstrap percolation) in $\Z^d$ with random initial configurations. Confirming a conjecture of Balister, Bollob\'as, Przykucki and Smith, who proved the corresponding result in two dimensions, we show that the critical probability is non-zero for all subcritical models.
\end{abstract}

\maketitle

\section{Introduction}

The study of bootstrap percolation, which may be thought of as a monotone version of the Glauber dynamics of the Ising model, was initiated in 1979 by Chalupa, Leath and Reich~\cite{CLR}. One of the most important early results was obtained by
Schonmann~\cite{Sch1}, who proved\footnote{In the case $d = 2$, this result was obtained several years earlier, by van Enter~\cite{vE}.} that the critical probability $p_c(\Z^d,r)$ of the $r$-neighbour model on $\Z^d$ (see below) satisfies
\[
p_c(\Z^d,r) =
\begin{cases}
0 & \text{if } r \leq d, \text{ and} \\
1 & \text{otherwise.}
\end{cases}
\]
In this paper we study the corresponding problem in a vastly more general setting, whose study was initiated in 2015 by Bollob\'as, Smith and Uzzell~\cite{BSU}.

\begin{definition}
Let $\U = \{X_1,\dots,X_m\}$ be an arbitrary finite collection of finite, non-empty subsets of $\Z^d \setminus \{\0\}$. Now, given a set $A \subset \Z^d$ of initially \emph{infected} sites, set $A_0 = A$, and define for each $t \in \N$ the set $A_t$ of sites infected at time $t$ by
\[
A_t = A_{t-1} \cup \big\{ x \in \Z^d : x + X \subset A_{t-1} \text{ for some } X \in \U \big\}.
\]
The \emph{$\U$-closure} of $A$ is the set $[A]_\U := \bigcup_{t \geq 0} A_t$ of all eventually-infected sites, and we say that $A$ \emph{percolates} if all sites are eventually infected; that is, if $[A]_\U = \Z^d$.
\end{definition}

We call\/ $\U$ the \emph{update family} of the process, each $X \in \U$ an \emph{update rule}, and the process itself\/ \emph{$\U$-bootstrap percolation}. Thus, according to the definition, a site $x$ becomes infected in a given step if the translate by $x$ of one of the sets of the update family is already entirely infected, and infected sites remain infected forever. For example, the classical $r$-neighbour model on $\Z^d$, mentioned above, is defined as the process in which a site becomes infected if at least $r$ of its neighbours is infected, and its update family $\cN_r^d$ consists of all $\binom{2d}{r}$ subsets of size $r$ of the $2d$ nearest neighbours of the origin.

We are interested in the behaviour of the $\U$-bootstrap process when the initial set of infected sites $A$ is chosen randomly. Let us say that a set $A \subset \Z^d$ is \emph{$p$-random} if each of the sites of $\Z^d$ is included in $A$
independently with probability~$p$, write $\P_p$ for the corresponding probability measure, and define the \emph{critical probability} to be\footnote{One can show using the $0$-$1$ law for translation-invariant events that the probability $A$ percolates is either 0 or 1, so the constant $1/2$ in the definition is not important.}
\begin{equation}\label{def:pc}
p_c(\Z^d,\U) := \inf \big\{ p \,:\, \P_p\big( [A]_\U = \Z^d \big) \geq 1/2 \big\}.
\end{equation}

One of the key insights from~\cite{BSU} was that, at least in two dimensions, the rough global behaviour of the $\U$-bootstrap process depends only on the action of the process on discrete half-spaces. In order to make this statement precise, let $\SS^{d-1}$ be the unit sphere in $\R^d$, and for each $u \in \SS^{d-1}$ let us write
\[
\H_u^d := \big\{ x \in \Z^d \,:\, \<x,u\> < 0 \big\}
\]
for the discrete half-space in $\Z^d$ with normal $u \in \SS^{d-1}$. Now, given a $d$-dimensional update family~$\U$, define
\[
\S = \S(\U) := \big\{ u\in \SS^{d-1} \,:\, [\H_u^d]_\U = \H_u^d \big\}
\]
to be the set of \emph{stable} directions, and note that $u$ is unstable if and only if $X \subset \H_u^d$ for some $X \in \U$. It is moreover easy to show that if $u$ is unstable then $[\H_u^d]_\U = \Z^d$.

The following definition was introduced by Bollob\'as, Smith and Uzzell~\cite{BSU} (when $d = 2$) and by Balister, Bollob\'as, Przykucki and Smith~\cite{BBPS} (for $d > 2$). Given a set $\T \subset \SS^{d-1}$, let $\interior(\T)$ denote the interior of $\T$ in the usual topology on the sphere $\SS^{d-1}$.

\begin{definition}\label{def:subcritical}
A $d$-dimensional update family is \emph{subcritical} if
\[
\interior(H\cap\S) \ne \emptyset
\]
for every hemisphere $H \subset \SS^{d-1}$.
\end{definition}

For example, the stable set of the $r$-neighbour model on $\Z^d$ has empty interior if $r \le d$, and is equal to $\SS^{d-1}$ otherwise, and is therefore subcritical if and only if $r > d$.

The following theorem was conjectured by Bollob\'as, Smith and Uzzell~\cite{BSU}, and proved by Bollob\'as, Smith and Uzzell~\cite{BSU} (for non-subcritical families) and by Balister, Bollob\'as, Przykucki and Smith~\cite{BBPS} (for subcritical families).

\begin{theorem}\label{thm:BBPS}
Let\/ $\U$ be a two-dimensional update family. Then
\[
p_c(\Z^2,\U) > 0 \qquad \Leftrightarrow \qquad \U \text{ is subcritical.}
\]
\end{theorem}

Balister, Bollob\'as, Przykucki and Smith~\cite{BBPS} moreover conjectured that the corresponding statement also holds for all $d > 2$. The main aim of this paper is to prove the following theorem, which confirms one direction of this conjecture. We remark that an alternative (very different) proof of this theorem 
has recently been given by Hartarsky and Szab\'o~\cite{HS22}. 

\begin{theorem}\label{thm:subcrit}
Let\/ $\U$ be a subcritical $d$-dimensional update family. Then
\[
p_c(\Z^d,\U) > 0.
\]
\end{theorem}

For $d$-dimensional update families that are not subcritical, the behaviour of the $\U$-bootstrap process is quite different, and controlling the growth of the infected set requires an essentially disjoint set of tools and techniques. For these models, the following much more precise `universality' conjecture was proposed by Bollob\'as, Duminil-Copin, Morris and Smith~\cite{BDMS}, and proved in~\cite{BBMSlower,BBMSupper} (the special case $d = 2$ was proved earlier by Bollob\'as, Smith and Uzzell~\cite{BSU}). Let $\log_{(r)}$ denote the $r$-times iterated logarithm, so $\log_{(0)} n = n$ and $\log_{(r)}n = \log \log_{(r-1)}n$ for each $r \geq 1$. 

\begin{theorem}\label{conj:highd}
Let\/ $\U$ be a $d$-dimensional update family. If\/ $\U$ is not subcritical, then $p_c(\Z^d,\U) = 0$, and moreover\/\footnote{Here $p_c(\Z_n^d,\U)$ is defined as in~\eqref{def:pc}, replacing $\Z^d$ by $\Z_n^d$.}
\[
p_c(\Z_n^d,\U) = \bigg( \frac{1}{\log_{(r-1)}n} \bigg)^{\Theta(1)}
\]
for some $r \in \{1,\dots,d\}$. 
\end{theorem}

Combining Theorems~\ref{thm:subcrit} and~\ref{conj:highd}, we obtain the following corollary, which confirms the conjecture of Balister, Bollob\'as, Przykucki and Smith~\cite{BBPS}.

\begin{corollary}\label{conj:BBPS}
Let\/ $\U$ be a $d$-dimensional update family. Then
\[
p_c(\Z^d,\U) > 0 \qquad \Leftrightarrow \qquad \U \text{ is subcritical.}
\]
\end{corollary}

We remark that moreover $p_c(\Z^d,\U) = 1$ if and only if $\S(\U) = \SS^{d-1}$. The proof of this assertion uses a technical lemma from~\cite{BBMSupper}; we provide the details in Section~\ref{sec:pc=1}. 

The non-triviality of $p_c(\Z^d,\U)$ for subcritical update families $\U$ means that one can ask of such models questions that would more typically be associated with (classical) percolation, including those concerning behaviour at criticality, the probabilities of one-arm events below criticality, and noise sensitivity. A number of such questions were asked of two-dimensional models in~\cite{BBPS}, and solutions to several of them were subsequently obtained by Hartarsky~\cite{Hart}. In dimensions $d \geq 3$, all such questions remain open. Since the questions are essentially the same in all dimensions, we do not repeat them here, but instead refer the reader to~\cite{BBPS}.

The proof of Theorem~\ref{thm:subcrit}, like the proof in~\cite{BBPS}, uses multi-scale analysis, and our main challenge will be to define suitable high-dimensional `covers' of our (random) set of infected sites. In order to handle the additional
complexities of high-dimensional geometry, we found it necessary to develop a new method that is somewhat simpler than the one used in~\cite{BBPS}, and which we call `pinching a hyperplane'.

\begin{HistRmk}
The results proved in this paper were first announced in 2017, but the proof of Theorem~\ref{thm:subcrit} was not written down carefully until early 2020. The proof that $p_c(\Z^d,\U) = 0$ for all non-subcritical update families $\U$, on the other hand, and hence also the proof of Corollary~\ref{conj:BBPS}, was completed only very recently, in~\cite{BBMSupper}.  
\end{HistRmk}

\section{An outline of the proof}\label{sec:outline}

In this section we give a high-level overview of the strategy we shall use to prove Theorem~\ref{thm:subcrit}. We shall in fact prove the theorem in the following quantitative form. 

\begin{theorem}\label{thm:subcrit2}
Let $\U$ be a subcritical $d$-dimensional update family. Then
\[
\P_p\big( \0 \in [A] \big) = O\big( p^{2/3} \big).
\]
In particular, $p_c(\Z^d,\U) > 0$.
\end{theorem}

We shall prove Theorem~\ref{thm:subcrit2} using a multi-scale argument. Hypercubes in $\R^d$, at increasing scales, are deemed either `good' or `bad' (see Definition~\ref{def:goodcubes}). At the smallest scale, a hypercube is `good' if its intersection with the $p$-random set $A$ is empty. Thereafter, a hypercube at the $k$th scale is `good' (roughly speaking) if it does not contain two `independent' bad hypercubes at the $(k-1)$th scale. The idea is that we can find a set of initially uninfected sites (or `barrier'), looking somewhat like a polytope whose sides have been perturbed to avoid nearby infected sites, around each `bad' hypercube at the $(k-1)$th scale that is contained in a `good' hypercube at the $k$th scale. Moreover, and crucially, the finite set of sites of $\Z^d$ bounded by that barrier (including the `bad' hypercube itself) is $\U$-closed\footnote{We say that a set $Q \subset \Z^d$ is $\U$-closed if $[Q]_\U = Q$.} (see Proposition~\ref{prop:exists-T}).

In this way we build up a sequence of barriers with the following properties: each barrier bounds a finite $\U$-closed set of sites; any pair of barriers (together with the sites bounded by them) are either disjoint or nested; and the union of all barriers and their interiors contains $A$, but is (almost surely) not all of $\Z^d$. We emphasize that all of the technical difficulties in the proof will occur during the (deterministic) construction of the barriers (that is, during the proof of Proposition~\ref{prop:exists-T}), which is carried out in Sections~\ref{sec:perturbed}--\ref{sec:covers}. Our only probabilistic argument is quite straightforward, and is given in Section~\ref{sec:cubes}. 

We shall use the fact that $\U$ is subcritical in order to construct approximately-polytopal $\U$-closed sets whose faces are perturbed locally so that they avoid nearby infected sites. Such sets exist because the normals to the faces are in `strongly
stable' directions.

\begin{definition}\label{def:stronglystable}
The interior $\interior(\S)$ of the set $\S = \S(\U)$ of stable directions of~$\U$ is called the \emph{strongly stable set} of $\U$. Directions $u \in \interior(\S)$ are called \emph{strongly stable}.
\end{definition}

Recall that a direction $u$ is stable if the half-space $\H_u^d$ is $\U$-closed; the advantage of strongly stable directions is that `small perturbations' of $\H_u^d$ are also $\U$-closed (see Lemma~\ref{lem:range-closed}). More precisely, this is true if $\U$ is restricted to \emph{destabilizing} rules, i.e., rules $X \in \U$ with $\0 \notin \conv(X)$, since non-destabilising rules could cause local infections in `valleys' on the surface of a perturbed half-space. To avoid this problem, we shall use (as the directions of the faces of our barriers) strongly stable directions that avoid the set
\begin{equation}\label{def:FU}
F(\U) := \bigcup_{X \in\, \U} \, \bigcup_{\substack{x,y \in X \\ x \ne y}} \big\{ u \in \SS^{d-1} \,:\, \< x - y, u \> = 0 \big\}
\end{equation}
of all $u$ perpendicular to a line joining any two sites in any update rule. Since the set $\F(\U)$ is nowhere dense in $\SS^{d-1}$, this restriction has a trivial effect on our choice, made in the following lemma, of strongly stable directions to use in the proof.

\begin{lemma}\label{lem:stronglyexists}
Let $\U$ be a subcritical $d$-dimensional update family. Then there exists a finite set $\str \subset \interior(\S) \setminus F(\U)$ such that $H \cap \str \ne \emptyset$ for every open hemisphere $H \subset \SS^{d-1}$.
\end{lemma}

\begin{proof}
We use the compactness of $\SS^{d-1}$. First, set $\S^\circ := \interior(\S) \setminus F(\U)$ and
\[
\HH := \big\{ H_u:u\in \S^\circ\}, \qquad \text{where} \quad H_u := \big\{ v \in \SS^{d-1} : \< u,v \> > 0 \big\},
\]
so that $\HH$ is the collection of all open hemispheres in $\SS^{d-1}$ centred at elements of~$\S^\circ$. We claim that $\HH$ is an open cover of $\SS^{d-1}$. To show this, let $w \in \SS^{d-1}$ and observe that $w \in H_u$ for some $H_u \in \HH$ if and only if $u \in H_w$ for some $u\in \S^\circ$. It therefore suffices to show that $H \cap \S^\circ$ is non-empty for every open hemisphere $H \subset \SS^{d-1}$. Thus, let $H \subset \SS^{d-1}$ be an open hemisphere, and recall from Definition~\ref{def:subcritical} that $\interior( H \cap \S ) \ne \emptyset$, so there exists a non-empty open set $\O \subset H \cap \interior(\S)$. Since $F(\U)$ is a finite union of $(d-2)$-dimensional subspheres of $\SS^{d-1}$, it is nowhere dense in $\SS^{d-1}$, and it follows that $\O \setminus F(\U)$ has non-empty interior. In particular, $H \cap \S^\circ$ is non-empty, as claimed. 

Now, since $\HH$ is an open cover of $\SS^{d-1}$, it follows that it has a finite sub-cover. Moreover, the set $\str$ of centres of the open hemispheres in this finite sub-cover has the desired property, since if $H \subset \SS^{d-1}$ is an open hemisphere with centre $w$, then $w \in H_u$ for some $u \in \str$, which implies that $u \in H \cap \str$, as required.
\end{proof}

It is natural to ask whether one can always choose $\str$ to have size $d + 1$. In fact one can always choose such an $\str$, and this can be shown using Helly's Theorem.\footnote{The authors thank Wojciech Samotij for pointing this out.} This is optimal, since $\S$ might consist of small open balls around the vertices of a regular $d$-dimensional simplex inscribed in $\SS^{d-1}$.


Let us fix, for the rest of the paper, a subcritical $d$-dimensional update family $\U$, a set $\str \subset \interior(\S) \setminus F(\U)$ as in Lemma~\ref{lem:stronglyexists}, and a constant $\eps > 0$ such that
\begin{equation}\label{eq:subcrit-eps}
\big\{ v \in \SS^{d-1} : \| u - v \| \le \eps \big\} \subset \interior(\S) \setminus F(\U)
\end{equation}
for each $u \in \str$, where we write (here and throughout the paper) $\| \cdot \|$ for the Euclidean norm on $\R^d$. We also write $d(x,y) = \| x - y \|$ and, for $x \in \R^d$ and $r \ge 0$, 
\begin{equation}\label{def:Br}
B_r(x) := \big\{ y \in \R^d \,:\, d(x,y) \le r \big\}
\end{equation}
for the closed Euclidean ball of radius $r$ centred at $x$. Given sets $\X,\Y \subset \R^d$, we write $d(\X,\Y)$ for the infimum of $d(x,y)$ over $x \in \X$ and $y \in \Y$, and say that $\X$ and $\Y$ are \emph{adjacent} if $d(\X,\Y) = 0$. Finally, we write $\X_\Z$ for the discrete set $\X \cap \Z^d$.

To finish the section, let us note that $A$ will always denote a subset of $\Z^d$; in all deterministic statements this set will be arbitrary, and in probabilistic statements it will be chosen to be $p$-random (i.e., we consider the product measure $\Pr_p$ on subsets of $\Z^d$). In particular, we shall use $A$ to define `good' and `bad' cubes (see Definition~\ref{def:goodcubes}), and thereby $A$ will appear in our main deterministic statement, Proposition~\ref{prop:exists-T}.

\section{Good and bad cubes, and the main proposition}\label{sec:cubes}

In this section we state our main deterministic result, Proposition~\ref{prop:exists-T}, and use it to deduce Theorem~\ref{thm:subcrit2}. The first step is to define explicitly the framework for our multi-scale argument. This will involve defining the hypercubes that we shall work with at each scale, and defining precisely `good' and `bad' hypercubes.

First, we define sequences $(\Delta_k)_{k=1}^\infty$, which will be the side-lengths of hypercubes at the $k$th scale, and $(g_k)_{k=1}^\infty$, which will be the maximum distance between a hypercube $Q$ at the $(k+1)$th scale, and a hypercube at the $k$th scale that can affect whether $Q$ is good or bad (see~\eqref{eq:k-cubes} and Definition~\ref{def:goodcubes}). These will need to be chosen so that $\Delta_k \ll g_k \ll \Delta_{k+1}$.

Thus, fix an arbitrary $1 < \beta < 3/2$, and let $p > 0$ be sufficiently small. Set $\Delta_1 := \big\lfloor p^{-1/(3d+2)} \big\rfloor$, and for each $k \ge 1$, define
\begin{equation}\label{def:Dk:gk}
\Delta_{k+1} := \big\lfloor \Delta_k^{1/2} \big\rfloor \cdot \Delta_k  \qquad \text{and} \qquad g_k := \Delta_k^{\beta}.
\end{equation}

Now, a \emph{$(k)$-cube} is a (continuous) subset of $\R^d$ of the form
\begin{equation}\label{eq:k-cubes}
x + \big[ 0, \Delta_k \big)^d,
\end{equation}
for some $x \in (\Delta_k \Z)^d$. Note in particular that the $(k)$-cubes form a tiling of $\R^d$. 

As noted in Section~\ref{sec:outline}, the following definition depends on the (arbitrary) set $A \subset \Z^d$.

\begin{definition}\label{def:goodcubes}
A (1)-cube $Q$ is \emph{good} if $Q \cap A = \emptyset$, and otherwise it is \emph{bad}. For each $k \ge 2$, a $(k)$-cube $Q$ is \emph{bad} if there exist non-adjacent bad $(k-1)$-cubes $Q_1$ and $Q_2$ with
\begin{equation}\label{eq:QQ1Q2}
\max \big\{ d(Q,Q_1), d(Q,Q_2) \big\} \le g_{k-1};
\end{equation}
otherwise $Q$ is \emph{good}. Note that $Q_1$ and $Q_2$ may lie outside~$Q$.
\end{definition}

If $Q$ is a $(k)$-cube and $k \geq 2$, then the event $\{ Q \text{ is good} \}$ depends on elements of $A$ outside of $Q$, and therefore these events are not (in general) independent for different $(k)$-cubes. This is why we allow collections of pairwise-adjacent bad $(k-1)$-cubes inside good $(k)$-cubes; it is also the reason, in the following definition, that we take maximal unions of pairwise-adjacent bad $(k)$-cubes, rather than singleton bad $(k)$-cubes.

\begin{definition}\label{def:Qk}
For each $k \geq 1$, define $\QQ_k = \QQ_k(A)$ to be the collection of all sets $Q \subset \R^d$ such that $Q$ is the union of a maximal collection of pairwise-adjacent bad $(k)$-cubes, and $Q$ intersects a good $(k+1)$-cube. For each $k \ge 1$ and each $Q \in \QQ_k$, let $x_Q$ be an arbitrary (but fixed) element of $Q$.
\end{definition}

Thus, if $Q \in \QQ_k$ then $Q = Q_1 \cup \dots \cup Q_\ell$, where $Q_1,\dots,Q_\ell$ are distinct bad $(k)$-cubes, $Q_i$ and $Q_j$ are adjacent for all $i \neq j$ (so, in particular, $1 \le \ell \le 2^d$), and $Q$ intersects a good $(k+1)$-cube. Moreover, since $Q$ is maximal, it follows from~\eqref{def:Dk:gk} and Definition~\ref{def:goodcubes} that all other bad $(k)$-cubes lie at distance at least $g_k / 2$ from $Q$.

We are now ready to state our main deterministic result, Proposition~\ref{prop:exists-T}, whose proof will take up Sections~\ref{sec:perturbed}--\ref{sec:covers}. Recall from Section~\ref{sec:outline} that our plan is to cover each cluster $Q$ of bad $(k)$-cubes that intersect a good $(k+1)$-cube by a set (surrounded by a `barrier') whose intersection with $\Z^d$ is $\U$-closed. The following proposition provides us with such a set, $T_k(Q)$, and moreover guarantees that this set is not too large.

\begin{prop}\label{prop:exists-T}
There exists $\gamma > 0$ depending only on $\str$ such that the following holds. For every set $A \subset \Z^d$, and for each $k \geq 1$ and $Q \in \QQ_k(A)$, there exists a set $T_k(Q)$, with $Q \subset T_k(Q) \subset B_{\gamma \Delta_k}(x_Q)$, such that $\T_\Z = \T \cap \Z^d$ is $\U$-closed, where
\begin{equation}\label{def:T}
\T := \bigcup_{k=1}^\infty \bigcup_{Q \in \QQ_k(A)} T_k(Q),
\end{equation}
\end{prop}

Let us now show that Theorem~\ref{thm:subcrit2} is a straightforward consequence of Proposition~\ref{prop:exists-T}. To do so, we prove first two simple lemmas about bad $(k)$-cubes. We then use these to show that $[A]_\U \subset \T$ almost surely, and to bound the probability that $\0 \in \T$.

\begin{lemma}\label{lem:k2ind}
For each $k \ge 1$, and every pair of non-adjacent $(k)$-cubes $Q$ and $Q'$, the events $\{ Q \text{ is bad} \}$ and $\{ Q' \text{ is bad} \}$ are independent with respect to the measure $\Pr_p$.
\end{lemma}

\begin{proof}\hspace{-0.15cm}\footnote{This proof corrects a small mistake in~\cite{BBPS}.}
We shall show that the events $\{ Q \text{ is bad} \}$ and $\{ Q' \text{ is bad} \}$ depend on (the intersection of $A$ with) disjoint subsets of $\Z^d$, which immediately implies that they are independent with respect to the product measure $\Pr_p$. To do so, note first that the state (either good or bad) of a $(k)$-cube $Q$ depends on the states of the $(k-1)$-cubes within distance $g_{k-1}$ of~$Q$. These in turn depend on the states of the $(k-2)$-cubes within distance $g_{k-2}$ of those $(k-1)$-cubes, and so on, until we reach $(1)$-cubes, whose states do not depend on any sites outside of them. Thus, if a site $x$ affects the state of~$Q$, then, by~\eqref{def:Dk:gk}, and since $1 < \beta < 3/2$ and $p$ is sufficiently small, the distance of $x$ from $Q$ must be at most
\begin{equation}\label{eq:sphere:of:influence}
\sum_{i = 1}^{k-1} \big( g_i + \sqrt{d} \cdot \Delta_i \big) \le \sum_{i = 1}^{k-1} 2 \cdot \Delta_i^\beta \le 3 \Delta_{k-1}^\beta < \Delta_k / 3.
\end{equation}
However, if $Q$ and $Q'$ are not adjacent, then their distance from each other is at least $\Delta_k$, and hence the sets of sites that affect their states are disjoint, as claimed.
\end{proof}

It is now easy to bound the probability that a $(k)$-cube is bad.

\begin{lemma}\label{lem:probibad}
For any $(k)$-cube $Q$,
\[
\P_p\big(\text{$Q$ is bad}\big) \le \Delta_k^{-(2d+2)}.
\]
\end{lemma}

\begin{proof}
The proof is by induction on~$k$. Set $q_k := \Delta_k^{-(2d+2)}$, and recall that a (1)-cube is bad if and only if it contains an element of $A$. Since $\Delta_1 \le p^{-1/(3d+2)}$, the expected size of the set $Q \cap A$ is $\Delta_1^d \cdot p \le \Delta_1^{-(2d+2)} = q_1$, so the claimed bound holds when $k = 1$. 

So let $k \ge 2$, let $Q$ be a $(k)$-cube, and suppose that the claimed bound holds for $(k-1)$-cubes. If $Q$ is bad then, by Definition~\ref{def:goodcubes}, there exist distinct, non-adjacent bad $(k-1)$-cubes $Q_1$ and $Q_2$ satisfying~\eqref{eq:QQ1Q2}. By~\eqref{def:Dk:gk}, and since $\beta < 3/2$, there are at most
\[
\bigg( \frac{3\Delta_k}{\Delta_{k-1}} \bigg)^d \le 3^d \cdot \Delta_{k-1}^{d/2}
\]
choices for each of $Q_1$ and $Q_2$, and the states of $Q_1$ and $Q_2$ are independent by Lemma~\ref{lem:k2ind}. 

\noindent It follows that
\[
\P_p\big(\text{$Q$ is bad}\big) \leq \big( 3^d \cdot \Delta_{k-1}^{d/2} \big)^2 \cdot q_{k-1}^2 = 3^{2d} \cdot \Delta_{k-1}^{-(3d+4)},
\]
and hence, since $p$ was chosen sufficiently small, and again using~\eqref{def:Dk:gk},
\[
\P_p\big( \text{$Q$ is bad} \big) \leq \Delta_{k-1}^{-(3d+3)} \leq \Delta_k^{-2(3d+3)/3} = q_k,
\]
as required.
\end{proof}

To deduce Theorem~\ref{thm:subcrit2} from Proposition~\ref{prop:exists-T}, we apply Lemma~\ref{lem:probibad} twice: first to show that $[A]_\U \subset \T$ almost surely, and then to bound the probability that $\0 \in \T$. 

\begin{proof}[Proof of Theorem~\ref{thm:subcrit2}]
We may assume $p$ is sufficiently small, otherwise the assertion holds trivially. Let $A$ be a $p$-random subset of $\Z^d$ and set $\QQ_k := \QQ_k(A)$. Now, for each $k \geq 1$ and $Q \in \QQ_k$, let $T_k(Q)$ be the set given by Proposition~\ref{prop:exists-T}, and let $\T \subset \R^d$ be defined as in~\eqref{def:T}. In particular, by the proposition, the set $\T_\Z \subset \Z^d$ is $\U$-closed. 

We claim that $[A]_\U \subset \T$ almost surely. To prove this, let $x \in A$ and consider the (unique) sequence $x \in Q_1 \subset Q_2 \subset \dots$ such that $Q_k$ is a $(k)$-cube. By Lemma~\ref{lem:probibad}, the probability $Q_k$ is bad tends to zero as $k \to \infty$, and hence almost surely some member of the sequence is good. Noting that $Q_1$ is bad (since $x \in A$), choose $k \ge 1$ minimal such that $Q_{k+1}$ is good, and observe that, by Definition~\ref{def:Qk}, the bad $(k)$-cube $Q_k$ is contained in some member of $\QQ_k$. It follows that $x$ is almost surely contained in $Q \subset T_k(Q)$ for some $k \ge 1$ and $Q \in \QQ_k$, and hence the set $A$ is almost surely contained in $\T$. But $\T_\Z$ is $\U$-closed, so if $A \subset \T$ then the closure $[A]_\U$ is also contained in $\T$, as claimed.

It follows from the claim, and the definition~\eqref{def:T} of $\T$, that
\begin{equation}\label{eq:subcrit-proof-union}
\P_p\big( \0 \in [A]_\U \big) \leq \, \P_p\bigg( \0 \in \bigcup_{k=1}^\infty \bigcup_{Q \in \QQ_k} T_k(Q) \bigg) \leq \, \sum_{k=1}^\infty \P_p\bigg( \0 \in \bigcup_{Q \in \QQ_k} T_k(Q) \bigg).
\end{equation}
To bound the right-hand side of~\eqref{eq:subcrit-proof-union}, recall from Definition~\ref{def:Qk} and Proposition~\ref{prop:exists-T} that $T_k(Q)$ contains at least one bad $(k)$-cube, and is contained in the ball $B_{\gamma \Delta_k}(x_Q)$. Thus, if $\0 \in T_k(Q)$ for some $Q \in \QQ_k$, then $\0 \in B_{\gamma \Delta_k}(x_Q)$, and so there must exist a bad $(k)$-cube within distance $2\gamma \Delta_k$ of~$\0$. Noting that there are at most $(4\gamma)^d$ such cubes, it follows, by Lemma~\ref{lem:probibad}, and since $\Delta_1 = \lfloor p^{-1/(3d+2)} \rfloor$ and $p$ is sufficiently small, that 
\[
\P_p\big( \0 \in [A]_\U \big) \le (4\gamma)^d \sum_{k=1}^\infty \Delta_k^{-(2d+2)} \le (8\gamma)^d \cdot \Delta_1^{-(2d+2)} = O\big( p^{(2d+2)/(3d+2)} \big),
\]
as required.
\end{proof}

In order to complete the proof of Theorem~\ref{thm:subcrit2}, it therefore suffices to prove Proposition~\ref{prop:exists-T}. To do so, first, in Section~\ref{sec:perturbed}, we define a family of `perturbed surfaces' that will be used to construct the boundaries of the sets $T_k(Q)$. Then, in Section~\ref{avoid:sec}, we show that these surfaces can be chosen to avoid bad cubes. Finally, in Section~\ref{sec:covers}, we use these surfaces to construct the sets $T_k(Q)$, and show that they have the claimed properties. We remark that most of the technical difficulties are contained in Section~\ref{avoid:sec}.

\section{Perturbed surfaces}\label{sec:perturbed}

In this section we define and prove key properties of certain families of surfaces in~$\R^d$. These surfaces will later be used as the faces of the perturbed polytopes $T_k(Q)$ that we shall construct (in the proof of Proposition~\ref{prop:exists-T}) around clusters of bad $(k)$-cubes. The surfaces are defined (in Definition~\ref{def:pinch}) relative to a co-dimension 1 hyperplane, which is modified by adding `bumps' at various scales, the bumps at larger scales being flatter and more spread out than the bumps at smaller scales.

We say that a set $Z \subset \R^d$ is \emph{$i$-separated} if $d(x,y) > g_i / 2$ for all distinct $x,y \in Z$, and that a $k$-tuple $(Z_1,\dots,Z_k)$ of subsets of $\R^d$ is \emph{$k$-separated} if the set $Z_i$ is $i$-separated for each $1 \leq i \leq k$. We write $\{u\}^\perp := \{v \in \R^d : \< u,v \> = 0 \}$ for the co-dimension~$1$ hyperplane with normal $u$, and we shall use the function $c\colon \R \to \R$, defined by
$$c(x) := (\cos x)^2 \cdot \1 \big[ |x| \leq \pi/2 \big],$$
which we note is differentiable everywhere.


\begin{definition}\label{def:pinch}
Let $k \geq 0$ and $u \in \SS^{d-1}$. A \emph{$(k,u)$-pinch} is a surface 
\[
\Sigma = \Sigma(u,\lambda;Z_1,\dots,Z_k) \subset \R^d
\]
defined by a real number $\lambda$ and 
a $k$-separated $k$-tuple $(Z_1, \dots, Z_k)$, where each $Z_i$ is a subset of $\{u\}^\perp$, as follows:
\[
\Sigma := \big\{ x + h(x)u  : x \in \{u\}^\perp \big\},
\]
where $h\colon\R^d \to \R$ is the \emph{height function}\footnote{Note that the sets $Z_i$ may be infinite; however, we show in Lemma~\ref{lem:pinched-height} that the assumption that $(Z_1, \dots, Z_k)$ is $k$-separated implies that $h(x)$ is finite. Note also that when $k = 0$ we have $h(x) = \lambda$, so in this case $\Sigma$ is just a translation of the co-dimension 1 hyperplane $\{u\}^\perp$.} 
\begin{equation}\label{def:h}
h(x) := \lambda + \sum_{i=1}^k 2^4 \gamma \Delta_i \sum_{z \in Z_i} c \bigg( \frac{2^5 \cdot d(x,z)}{g_i} \bigg),
\end{equation}
Here $\gamma > 0$ is a sufficiently large constant depending only on $\str$. 
\end{definition}

\begin{figure}[ht]
\centering
\includegraphics{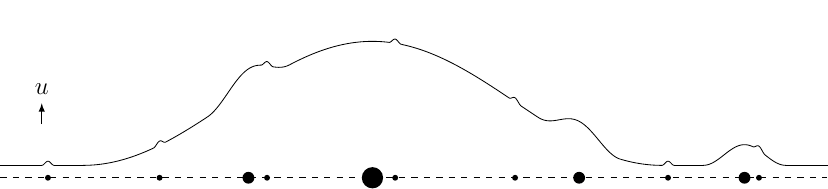}
\caption{A $(3,u)$-pinch with $d=2$. The dashed line along the bottom of the figure is the base of the pinch, somewhat vertically offset for clarity. Projected onto the base, the large dot is the element of $Z_3$, the medium dots are the elements of $Z_2$, and the small dots are the elements of $Z_1$.}
\label{fig:construct-pinch}
\end{figure}

We also define a corresponding \emph{$(k,u)$-range} $\Xi = \Xi(u, \lambda; Z_1, \dots, Z_k)$ by
$$\Xi := \big\{ x + hu \,:\, h < h(x), \, x \in \{u\}^\perp \big\}.$$
The constant $\gamma = \gamma(\str)$ in Definition~\ref{def:pinch} is the same as the constant $\gamma$ in Proposition~\ref{prop:exists-T}. It will be assumed to satisfy $\gamma > 2\sqrt{d} + 1$, and also
\begin{equation}\label{eq:gamma}
\bigcap_{u \in \str} \big\{ x \in \R^d : \< x,u \> \le 4d \big\} \subset B_\gamma(\0).
\end{equation}
Such a $\gamma$ exists because $H \cap \str \ne \emptyset$ for every open hemisphere $H \subset \SS^{d-1}$, by Lemma~\ref{lem:stronglyexists}.

We shall show, in Lemma~\ref{lem:range-closed}, that the set $\Xi_\Z$ is $\U$-closed for every $u \in \str$, every $k \ge 1$, and every $(k,u)$-range $\Xi$. However, in order to do so we first need to prove some simple properties of the \emph{partial height functions}
\begin{equation}\label{def:hj}
h_j(x) := \lambda + \sum_{i=j}^k 2^4 \gamma \Delta_i \sum_{z \in Z_i} c\bigg( \frac{2^5 \cdot d(x,z)}{g_i} \bigg),
\end{equation}
where $1 \leq j \leq k$. For convenience, define also $h_{k+1}(x) := \lambda$, and note that $h(x) = h_1(x)$. We remark that these properties will also be useful in Sections~\ref{avoid:sec} and~\ref{sec:covers}. 

In the proofs below, we refer to the co-dimension 1 hyperplane 
$$\Sigma(u,\lambda,\emptyset,\dots,\emptyset) = \big\{ x + \lambda u  \,:\, x \in \{u\}^\perp \big\},$$ 
as the \emph{base} of $\Sigma$ (or $\Xi$), and to the $k$-tuple $(Z_1,\dots,Z_k)$ as the \emph{augmentation} of $\Sigma$ (or $\Xi$). 

\begin{lemma}\label{lem:pinched-height}
Let $\Sigma$ be a $(k,u)$-pinch. For each $1 \leq j \leq k$, the partial height function $h_j$ of $\Sigma$ satisfies
\begin{equation}\label{eq:pinched-height-1}
\| h_j - h_{j+1} \|_\infty \le 2^4 \gamma \Delta_j \quad \text{and} \quad \| h_j - \lambda \|_\infty \leq 2^5 \gamma \Delta_k. 
\end{equation}
Moreover, if $x,y \in \R^d$, then
\begin{equation}\label{eq:pinched-height-2}
| h_j(x) - h_j(y) | \leq 2^{10} \gamma \Delta_j^{1 - \beta} \cdot \| x - y \|.
\end{equation}
\end{lemma}

\begin{proof}
Recall from Definition~\ref{def:pinch} that the augmentation $(Z_1,\dots,Z_k)$ of $\Sigma$ is $k$-separated, and therefore $d(y,z) > g_i / 2$ for every $i \in [k]$ and all distinct $y,z \in Z_i$. It follows that, for each $x \in \R^d$, there is at most one $z \in Z_i$ such that $d(x,z) \le g_i/4$, and hence at most one $z \in Z_i$ such that
\[
\frac{2^5 \cdot d(x,z)}{g_i} \leq \frac{\pi}{2}.
\]
The inequality $| h_j(x) - h_{j+1}(x) | \le 2^4 \gamma \Delta_j$ now follows immediately from~\eqref{def:hj}, and
\[
| h_j(x) - \lambda | = |h_j(x) - h_{k+1}(x)| \leq \sum_{i=j}^k 2^4 \gamma \Delta_i \le 2^5 \gamma \Delta_k
\]
also follows, since $h_{k+1}(x) = \lambda$, and by the triangle inequality and the definition~\eqref{def:Dk:gk} of~$\Delta_i$. Since both inequalities hold for all $x \in \R^d$, this proves~\eqref{eq:pinched-height-1}.  

To prove~\eqref{eq:pinched-height-2}, recall that the function $c$ is differentiable, and observe that 
\[
\bigg\| \nabla c\bigg( \frac{2^5 \cdot d(x,z)}{g_i} \bigg) \bigg\| \leq \frac{2^5}{g_i}.
\]
It follows that
\[
\| \nabla h_j(x) \| \leq \sum_{i=j}^k 2^4 \gamma \Delta_i \cdot \frac{2^5}{g_i} = \sum_{i=j}^k 2^9 \gamma \Delta_i^{1 - \beta} \leq 2^{10} \gamma \Delta_j^{1 - \beta},
\]
by the definitions~\eqref{def:Dk:gk} of $\Delta_i$ and $g_i$, and since $\beta > 1$ and $p$ is sufficiently small. Applying the mean value theorem for multivariate functions now yields~\eqref{eq:pinched-height-2}.
\end{proof}

We are now ready to prove the key property of $(k,u)$-ranges: they are $\U$-closed.

\begin{lemma}\label{lem:range-closed}
If $u \in \str$, then $\Xi_\Z$ is $\U$-closed for every $(k,u)$-range $\Xi$.
\end{lemma}

\begin{proof}
We are required to show that $x + X \not\subset \Xi_\Z$ for every $x \in \Z^d \setminus \Xi_\Z$ and $X \in \U$. To do so, we use the definition of $\str$, which was chosen using Lemma~\ref{lem:stronglyexists}, and our assumption that $p$ is small, which (by Lemma~\ref{lem:pinched-height}) implies that the fluctuations of the surface of $\Xi$ are small compared with $\eps$. 

Without loss of generality we may assume that $x = \0$, so suppose that $h(\0) \leq 0$, and that $X \subset \Xi_\Z$ for some $X \in \U$. We claim that
\begin{equation}\label{eq:quasiconvexunion}
\Xi_\Z \subset \bigcup_{v \in S_\eps(\SS^{d-1},u)} \H_v^d,
\end{equation}
where $S_\eps(\SS^{d-1},u) \subset \SS^{d-1}$ is the $(d-2)$-sphere consisting of points of $\SS^{d-1}$ at geodesic 
distance $\eps$ from $u$. To prove~\eqref{eq:quasiconvexunion}, observe first that, by Lemma~\ref{lem:pinched-height}, we have
\[
|h(x) - h(\0)| = |h_1(x) - h_1(\0)| \leq 2^{10} \gamma \Delta_1^{1 - \beta} \cdot \| x \|
\]
for each $x \in \R^d$. Since $h(\0) \leq 0$, and recalling that $\Delta_1 = \lfloor p^{-1/(3d+2)} \rfloor$ and $\beta > 1$, it follows that $h(x) \leq o( \| x \| )$ as $p \to 0$, uniformly over $x \in \R^d$. In particular, we may assume that $h(x) \leq \| x \| (\sin \eps) / 2$. Now, given $x \in \{u\}^\perp$, set $v := (\cos \eps) u - ( \sin \eps ) x / \|x\|$, which is an element of $S_\eps(\SS^{d-1},u)$, and observe that
\[
\< x + h(x)u, v \> = h(x) \cos \eps - \| x \| \sin \eps < 0.
\]
This completes the proof of~\eqref{eq:quasiconvexunion}.


It suffices, therefore, to show that no update rule $X \in \U$ is contained in the set on the right-hand side of~\eqref{eq:quasiconvexunion}. Observe first that if $v \in S_\eps(\SS^{d-1},u)$ then $\|u - v\| \le \eps$, and therefore $v \in \S$, by~\eqref{eq:subcrit-eps}, since $u \in \str$. It follows that $X \not\subset \H_v^d$ for each $v \in S_\eps(\SS^{d-1},u)$. We claim that if $X \subset \Xi_\Z$, then there exist $x_1,x_2 \in X$ and $v_1,v_2 \in S_\eps(\SS^{d-1},u)$ such that 
\begin{equation}\label{eq:xs:and:vs}
x_1 \in \H_{v_1}^d \setminus \H_{v_2}^d \qquad \text{and} \qquad x_2  \in \H_{v_2}^d \setminus \H_{v_1}^d.
\end{equation}
To prove this, choose $v_1 \in S_\eps(\SS^{d-1},u)$ such that $|X \cap \H_{v_1}^d|$ is maximal (recalling that $X$ is finite). Since $X \not\subset \H_{v_1}^d$, there must exist $x_2 \in X \setminus \H_{v_1}^d$. Moreover, if $X \subset \Xi_\Z$ then by~\eqref{eq:quasiconvexunion} there must exist $v_2 \in S_\eps(\SS^{d-1},u)$ with $x_2\in \H_{v_2}^d$. By the maximality of $|X \cap \H_{v_1}^d|$, it follows that there exists $x_1 \in X \cap \H_{v_1}^d$ with $x_1\notin \H_{v_2}^d$, as claimed.

To complete the proof, we shall deduce from~\eqref{eq:xs:and:vs} that there exists a direction $v \in F(\U)$ within distance $\eps$ of $u$, contradicting~\eqref{eq:subcrit-eps}. In order to guarantee that $v \in F(\U)$, we shall choose $v \in \SS^{d-1}$ with $\< x_1 - x_2, v \> = 0$, and in order to guarantee that $\| u - v \| \leq \eps$, we shall choose it on the geodesic in $\SS^{d-1}$ joining $v_1$ to $v_2$. The vector satisfying these two conditions is\footnote{Indeed, $\< x_1 - x_2, v_2 \> \<v_1,x_1 \> + \< x_2 - x_1, v_1 \> \< v_2, x_1 \> = \< x_1 - x_2, v_2 \> \<v_1,x_2 \> + \< x_2 - x_1, v_1 \> \< v_2, x_2 \>$.}
$$v := \frac{ \< x_1 - x_2, v_2 \> v_1 + \< x_2 - x_1, v_1 \> v_2 }{\| \< x_1 - x_2, v_2 \> v_1 + \< x_2 - x_1, v_1 \> v_2 \|},$$
where it follows from~\eqref{eq:xs:and:vs} that $\< x_1 - x_2, v_2 \> \ge 0$ and $\< x_2 - x_1, v_1 \> \ge 0$. Since $v$ is the projection onto $\SS^{d-1}$ of an element of the convex hull of $v_1$ and~$v_2$, it follows that $\| u - v \| \le \eps$, and by~\eqref{eq:subcrit-eps} this contradicts the fact that $v \in F(\U)$, 
as required.
\end{proof}

\section{Construction of pinches avoiding bad cubes}\label{avoid:sec}

In the previous section we defined $(k,u)$-pinches and $(k,u)$-ranges, and proved in Lemma~\ref{lem:range-closed} that $(k,u)$-ranges are $\U$-closed when $u \in \str$. In this section we shall show how to construct $(k,u)$-pinches that avoid (with room to spare) all bad $(i)$-cubes (for all $1 \le i \le k$) inside a region of good $(k+1)$-cubes.

In order to state Lemma~\ref{lem:construct-pinch}, which is the main result of this section, we need to introduce a little notation. Given $u \in \SS^{d-1}$ and $k \in \N$, we define the line segment
\begin{equation}\label{eq:buffer}
L_u^{(k)} := \big\{ \lambda u \,:\, |\lambda| \leq \Delta_k \big\},
\end{equation}
and recall the definition of the Minkowski sum $\X + \Y := \{ x+y : x \in \X, \, y \in \Y \}$.

The following lemma is the key step in the proof of Theorem~\ref{thm:subcrit2}. 

\begin{lemma}\label{lem:construct-pinch}
Let $k \geq 0$ and $u \in \SS^{d-1}$, and let $\Pi$ be a translation of $\{u\}^\perp$. If every $(k+1)$-cube intersecting $\Pi + L_u^{(k+1)}$ is good, then there exists a $(k,u)$-pinch $\Sigma$, with base $\Pi$, such that for each $1 \leq i \leq k$, every $(i)$-cube intersecting $\Sigma + 3\gamma \cdot L_u^{(i)}$ is good.
\end{lemma}

In order to prove Lemma~\ref{lem:construct-pinch}, we must construct a $k$-separated $k$-tuple $(Z_1, \dots, Z_k)$, where each $Z_i$ is a subset of $\{u\}^\perp$. We shall construct the sets $Z_i$ inductively, using the following lemma, which is really the heart of the matter. 

\begin{lemma}\label{lem:construct-pinch-induction}
Let $k \in \N$ and $u \in \SS^{d-1}$, let $1 \le i \leq k$, and let
\begin{equation}\label{eq:Sigma}
\Sigma = \Sigma(u,\lambda;\emptyset,\dots,\emptyset,Z_{i+1},\dots,Z_k)
\end{equation}
be a $(k,u)$-pinch. Suppose that all $(i+1)$-cubes intersecting $\Sigma + L_u^{(i+1)}$ are good.
Then there exists an $i$-separated set $Z_i \subset \{u\}^\perp$ such that if
\begin{equation}\label{eq:Sigma'}
\Sigma' = \Sigma(u,\lambda;\emptyset,\dots,\emptyset,Z_i,Z_{i+1},\dots,Z_k),
\end{equation}
then all $(i)$-cubes intersecting $\Sigma' + 4\gamma  \cdot L_u^{(i)}$ are good.
\end{lemma}

\begin{figure}[ht]
 \centering
 \includegraphics{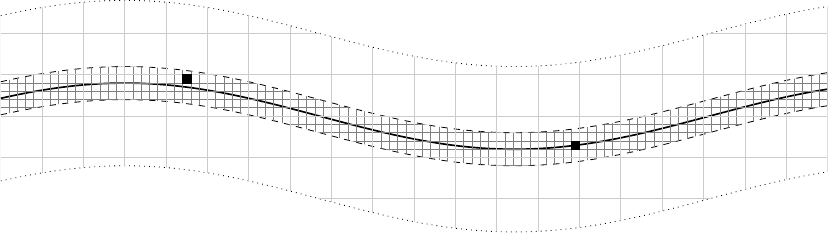}
 \includegraphics{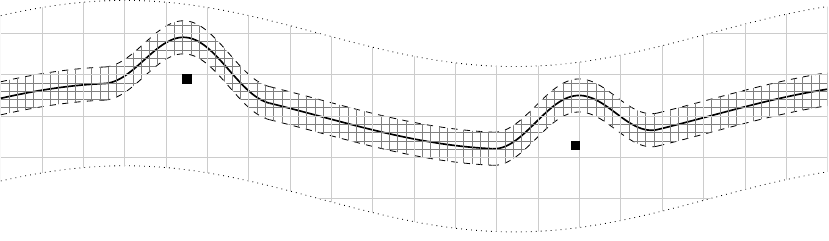}
 \caption{The $(k,u)$-pinch $\Sigma$, given by~\eqref{eq:Sigma}, is the solid black curve through the centre of the upper figure. The set $\Sigma + L_u^{(i+1)}$, bounded by the dotted lines, avoids all bad $(i+1)$-cubes. Also shown are two bad $(i)$-cubes, which lie inside $\Sigma + 4\gamma \cdot L_u^{(i)}$, the set bounded by the dashed lines. Lemma~\ref{lem:construct-pinch-induction} constructs $\Sigma'$, of the form~\eqref{eq:Sigma'}, and shown as the solid black line through the lower figure, such that $\Sigma' + 4\gamma \cdot L_u^{(i)}$ avoids all bad $(i)$-cubes.}
 \label{fig:pinch-induction}
\end{figure}

The idea of the proof is as follows. In order to construct $Z_i$, we take a point from each maximal collection of pairwise-adjacent bad $(i)$-cubes whose union intersects $\Sigma + 4\gamma  \cdot L_u^{(i)}$, and project those points (orthogonally) onto~$\{u\}^\perp$. This pulls the surface $\Sigma'$ away (locally) from the bad $(i)$-cubes, causing it to divert around them. We use the assumption that all $(i+1)$-cubes intersecting $\Sigma + L_u^{(i+1)}$ are good to show that this set is $i$-separated, and then again to show that every $(i)$-cube that intersects $\Sigma' + 4\gamma  \cdot L_u^{(i)}$ is good.

\begin{proof}[Proof of Lemma~\ref{lem:construct-pinch-induction}]
Let us choose, for each maximal collection of pairwise-adjacent bad $(i)$-cubes whose union $P$ intersects $\Sigma + 4\gamma  \cdot L_u^{(i)}$, an arbitrary point 
\begin{equation}\label{eq:choosing:xP}
y_P \in P \cap \big( \Sigma + 4\gamma  \cdot L_u^{(i)} \big).
\end{equation} 
Let $Y$ be the set of all such points $y_P$, and let $Z_i$ be the orthogonal projection of $Y$ onto~$\{u\}^\perp$. We shall prove, in the next two claims, that $Z_i$ has the required properties.

\begin{claim}\label{clm:ind1}
$Z_i$ is $i$-separated.
\end{claim}

\begin{clmproof}{clm:ind1}
Let $x,y \in Y$ with $x \ne y$, and let $a \in \R$ and $z \in \{u\}^\perp$ be such that $x - y = a u + z$, so  
\begin{equation}\label{eq:xyaz}
a = \< x - y, u \> \qquad \text{and} \qquad \|z\| \ge d(x,y) - |a|.
\end{equation}
We shall show that $d(x,y) \ge g_i$ and $|a| \le ( \| z \| + g_i) /4$, and hence $\| z \| > g_i / 2$.

To bound $d(x,y)$, we use our assumption that all $(i+1)$-cubes intersecting $\Sigma + L_u^{(i+1)}$ are good. Note first that $x \in \Sigma + L_u^{(i+1)}$, since $4\gamma  \cdot L_u^{(i)} \subset L_u^{(i+1)}$, by~\eqref{def:Dk:gk} and~\eqref{eq:buffer}. Therefore, the $(i+1)$-cube $Q$ containing $x$ is good. Moreover, the $(i)$-cubes $Q_1$ and $Q_2$ containing $x$ and $y$ (respectively) are both bad, since $x$ and $y$ are each contained in unions of pairwise-adjacent bad $(i)$-cubes. Since $Q_1 \subset Q$, it follows from Definition~\ref{def:goodcubes} that either $d(Q,Q_2) > g_i$, or $Q_1$ and $Q_2$ are adjacent. 

If $d(Q,Q_2) > g_i$ then we are done, since $x \in Q$ and $y \in Q_2$, so $d(x,y) \ge d(Q,Q_2)$. On the other hand, if $Q_1$ and $Q_2$ are adjacent then, since $x$ and $y$ are contained in distinct maximal collections of pairwise-adjacent bad $(i)$-cubes, $P_1$ and $P_2$, it follows that there exist bad $(i)$-cubes $Q_1' \in P_1$ and $Q_2' \in P_2$ such that $Q_1'$ and $Q_2'$ are non-adjacent. Since, for each $j \in \{1,2\}$, the $(i)$-cubes $Q_j$ and $Q_j'$ are adjacent (by definition of $P_j$), we have
\[
d(Q,Q'_j) \leq d(Q,Q_1) + \sqrt{d} \cdot \Delta_i + d(Q_1,Q_j') \le d(Q,Q_1) + 2 \sqrt{d} \cdot \Delta_i = 2 \sqrt{d} \cdot \Delta_i < g_i.
\]
Here, the second step holds because $Q_1$ is adjacent to $Q_1'$ and to $Q_2$, and $Q_2$ is adjacent to $Q_2'$, the third because $Q_1 \subset Q$, and the fourth by~\eqref{def:Dk:gk}. By Definition~\ref{def:goodcubes}, this contradicts our assumption that $Q$ is good, and so proves that $d(x,y) \ge g_i$, as claimed. 

In order to bound $|a|$, let $x'$ and $y'$ be (respectively) the orthogonal projections of $x$ and $y$ onto~$\{u\}^\perp$, and recall that $x,y \in \Sigma + 4\gamma  \cdot L_u^{(i)}$. It follows, by~\eqref{eq:buffer}, that 
$$|\< x - y, u \> | \le |h_{i+1}(x') - h_{i+1}(y')| + 8\gamma  \cdot \Delta_i.$$
Using Lemma~\ref{lem:pinched-height} to bound $|h_{i+1}(x') - h_{i+1}(y')|$, it follows that 
\begin{equation}\label{eq:a:bound}
|a| = | \< x - y, u \> | \le 2^{10} \gamma \Delta_{i+1}^{1 - \beta} \cdot \| x' - y' \| + 8\gamma  \cdot \Delta_i \le \frac{\|z\| + g_i}{4},
\end{equation}
where the final inequality holds since $\| x' - y' \| = \|z\|$ and by~\eqref{def:Dk:gk}, recalling that $\beta > 1$. 

We have shown that $d(x,y) \ge g_i$ and $|a| \le ( \| z \| + g_i) /4$, and it follows that
$$\| z \| \ge d(x,y) - |a| \ge \frac{3g_i}{4} - \frac{\| z \|}{4},$$
and hence $\| z \| > g_i / 2$. Since $x$ and $y$ were arbitrary elements of $Y$, it follows that $Z_i$ is $i$-separated, as claimed.
\end{clmproof}

Recall that the $(k,u)$-pinch $\Sigma'$ was defined in~\eqref{eq:Sigma'}. To complete the proof of the lemma, it only remains to prove the following claim.

\begin{claim}\label{clm:ind2}
Every $(i)$-cube intersecting $\Sigma' + 4\gamma  \cdot L_u^{(i)}$ is good.
\end{claim}

\begin{clmproof}{clm:ind2}
Suppose, for a contradiction, that $Q_1$ is a bad $(i)$-cube that intersects $\Sigma' + 4\gamma  \cdot L_u^{(i)}$, and let $P$ be the union of the maximal collection of pairwise-adjacent bad $(i)$-cubes containing $Q_1$. Suppose first that $P$ also intersects $\Sigma + 4\gamma  \cdot L_u^{(i)}$, in which case there exists $z \in Z_i$ such that $z$ is the orthogonal projection of $y_P$ onto $\{u\}^\perp$.

Let $x,y \in P$, with $x \in \Sigma' + 4\gamma  \cdot L_u^{(i)}$ and $y \in \Sigma + 4\gamma  \cdot L_u^{(i)}$, and observe that 
\begin{equation}\label{eq:ind2:upper}
|\< x - y, u \>| \le 2\sqrt{d} \cdot \Delta_i,
\end{equation}
by the definition of $P$. On the other hand, we have
\begin{equation}\label{eq:ind2:lower1}
\< x - y, u \> \ge h_i(x') - h_{i+1}(y') - 8\gamma  \cdot \Delta_i,
\end{equation}
where $x'$ and $y'$ are (as in the proof of Claim~\ref{clm:ind1}) the orthogonal projections of $x$ and $y$ onto~$\{u\}^\perp$, since $x \in \Sigma' + 4\gamma  \cdot L_u^{(i)}$ and $y \in \Sigma + 4\gamma  \cdot L_u^{(i)}$, and by Definition~\ref{def:pinch} and~\eqref{eq:buffer}. 

Now, since $x'$ and $z$ are both orthogonal projections of points of $P$ onto~$\{u\}^\perp$, we have $d(x',z) \le 2\sqrt{d} \cdot \Delta_i < g_i / 4$. Since $Z_i$ is $i$-separated, it follows that for every $z \ne z' \in Z_i$ we have $d(x',z') > g_i/4$, and hence $c\big( 2^5 \cdot d(x',z')/g_i \big) = 0$ (cf.~the proof of Lemma~\ref{lem:pinched-height}). Thus, by~\eqref{def:hj},
\begin{equation}\label{eq:ind2:lower2}
h_i(x') - h_{i+1}(x') = 2^4 \gamma \Delta_i \cdot c\bigg( \frac{2^5 \cdot d(x',z)}{g_i} \bigg) \ge 9\gamma \cdot \Delta_i,
\end{equation}
where the final inequality follows from $d(x',z) \le 2\sqrt{d} \cdot \Delta_i$ and~\eqref{def:Dk:gk}. Moreover, by Lemma~\ref{lem:pinched-height}, we have 
\begin{equation}\label{eq:ind2:lower3}
|h_{i+1}(x') - h_{i+1}(y')| \le 2^{10} \gamma \Delta_{i+1}^{1 - \beta} \cdot \| x' - y' \| \le \Delta_i,
\end{equation}
where the last inequality holds since $x'$ and $y'$ are both orthogonal projections of points of $P$ onto~$\{u\}^\perp$, so $\| x' - y' \| \le 2\sqrt{d} \cdot \Delta_i$, and recalling that $\beta > 1$. 

Thus, combining~\eqref{eq:ind2:lower1},~\eqref{eq:ind2:lower2} and~\eqref{eq:ind2:lower3}, it follows that $\< x - y, u \> \ge (\gamma - 1) \Delta_i$, which contradicts~\eqref{eq:ind2:upper} since $\gamma$ was chosen to be sufficiently large. This contradiction proves that $P$ cannot intersect $\Sigma + 4\gamma  \cdot L_u^{(i)}$.

So suppose now that $P \cap \big( \Sigma + 4\gamma  \cdot L_u^{(i)} \big) = \emptyset$, which means that there is no element of $Z_i$ corresponding to $P$. In this case we again use our assumption that all $(i+1)$-cubes intersecting $\Sigma + L_u^{(i+1)}$ are good, this time to obtain a contradiction. 

To begin, recall that $Q_1$ is a bad $(i)$-cube (contained in $P$) that intersects $\Sigma' + 4\gamma  \cdot L_u^{(i)}$, and let $Q$ be the $(i+1)$-cube containing $Q_1$. Observe that 
\[
\Sigma' + 4\gamma  \cdot L_u^{(i)} \subset \Sigma + L_u^{(i+1)},
\]
since $\| h_i - h_{i+1} \|_\infty \le 2^4 \gamma \cdot \Delta_i$, by Lemma~\ref{lem:pinched-height}, and $(2^4 + 4)\gamma \cdot \Delta_i \le \Delta_{i+1}$, by~\eqref{def:Dk:gk}. Since $Q_1 \subset Q$, it follows that $Q$ intersects $\Sigma + L_u^{(i+1)}$. Therefore, by Definition~\ref{def:goodcubes}, it suffices to show that there exists a bad $(i)$-cube $Q_2$, not adjacent to $Q_1$, with $d(Q,Q_2) \le g_i$. 

To do so, let $x \in Q_1 \cap \big( \Sigma' + 4\gamma  \cdot L_u^{(i)} \big)$, and observe that $h_i(x') \neq h_{i+1}(x')$, where $x'$ is the projection of $x$ onto $\{u\}^\perp$, since $x \in Q_1 \subset P$, and we assumed that $P$ does not intersect $\Sigma + 4\gamma  \cdot L_u^{(i)}$. It follows, by~\eqref{def:hj}, that there exists $z \in Z_i$ with $d(x',z) < 2^{-6} \pi \cdot g_i$. Let $y \in Y$ be such that $z$ is the orthogonal projection of $y$ onto~$\{u\}^\perp$, and let $Q_2$ be the $(i)$-cube containing $y$. Recall from~\eqref{eq:choosing:xP} that $Q_2$ is bad, and that $y \in \Sigma + 4\gamma  \cdot L_u^{(i)}$.

We now claim that 
\begin{equation}\label{eq:ind2-QQ2}
d(Q,Q_2) \le d(x,y) \le d(x',z) + |\< x - y, u \>| \le g_i.
\end{equation}
The first step holds since $x \in Q_1 \subset Q$ and $y \in Q_2$, and the second since $x'$ and $z$ are the orthogonal projections of $x$ and $y$ onto~$\{u\}^\perp$. Since $d(x',z) < 2^{-6} \pi \cdot g_i$, it is enough to show that $|\< x - y, u \>| \le g_i/2$. This follows since $x \in \Sigma' + 4\gamma  \cdot L_u^{(i)}$ and $y \in \Sigma + 4\gamma  \cdot L_u^{(i)}$, so
$$|\< x - y, u \>| \le |h_i(x') - h_{i+1}(z)| + 8\gamma  \cdot \Delta_i,$$
and by Lemma~\ref{lem:pinched-height} we have
\begin{align*}
|h_i(x') - h_{i+1}(z)| & \, \le \, |h_i(x')- h_{i+1}(x')| + |h_{i+1}(x') - h_{i+1}(z)| \\
& \, \le \, 2^4 \gamma \Delta_i + 2^{10} \gamma \Delta_{i+1}^{1 - \beta} \cdot \| x' - z \| \le g_i/4,
\end{align*}
where in the final step we used the bounds $\| x' - z \| = d(x',z) < 2^{-6} \pi \cdot g_i$ and $\beta > 1$. It follows that $|\< x - y, u \>| \le g_i/4 + 8 \gamma  \cdot \Delta_i \le g_i/2$, and so~\eqref{eq:ind2-QQ2} holds as claimed. 

We have shown that $Q_1$ and $Q_2$ are bad $(i)$-cubes, with $Q_1 \subset Q$ and $d(Q,Q_2) \le g_i$. Moreover, $Q_1 \ne Q_2$, since $Q_2$ intersects $\Sigma + 4\gamma  \cdot L_u^{(i)}$, and so is not contained in $P$. Therefore, if $Q_1$ and $Q_2$ are non-adjacent $(i)$-cubes, then $Q$ is a bad $(i+1)$-cube that intersects $\Sigma + L_u^{(i+1)}$, and we have the claimed contradiction. 

Finally, suppose that $Q_1$ and $Q_2$ are adjacent bad $(i)$-cubes. Then, since $P$ is maximal and $Q_2 \not\subset P$, there exists a bad $(i)$-cube $Q_1' \subset P$ that is not adjacent to $Q_2$. Since $d(Q_1',Q) \le d(Q_1',Q_1) = 0$, we again deduce that $Q$ is bad, as required.
\end{clmproof}

By Claims~\ref{clm:ind1} and~\ref{clm:ind2}, the set $Z_i \subset \{u\}^\perp$ is $i$-separated, and every $(i)$-cube intersecting $\Sigma' + 4\gamma  \cdot L_u^{(i)}$ is good, so the lemma follows.
\end{proof}

We may now complete the proof of Lemma~\ref{lem:construct-pinch} via a straightforward induction. 

\begin{proof}[Proof of Lemma~\ref{lem:construct-pinch}]
If $k = 0$ then we may take $\Sigma = \Pi$ and there is nothing to prove, so suppose that $k \ge 1$. We claim first that there exists a $(k,u)$-pinch
\[
\Sigma = \Sigma(u,\lambda;Z_1,\dots,Z_k),
\]
with base $\Pi$, such that for every $1 \leq i \leq k$, every $(i)$-cube intersecting $\Sigma_i + 4\gamma  \cdot L_u^{(i)}$ is good, where
\[
\Sigma_i := \Sigma(u,\lambda;\emptyset,\dots,\emptyset,Z_i,\dots,Z_k).
\]
We choose the sets $Z_1,\dots,Z_k$ in reverse order, inductively, using Lemma~\ref{lem:construct-pinch-induction}. For the base case of the induction, when $i = k$, we use our assumption that every $(k+1)$-cube intersecting $\Pi + L_u^{(k+1)}$ is good. By Lemma~\ref{lem:construct-pinch-induction}, it follows that there exists a $k$-separated set $Z_k \subset \{u\}^\perp$ such that every $(k)$-cube intersecting $\Sigma_k + 4\gamma  \cdot L_u^{(k)}$ is good. For the induction step, assume that every $(i+1)$-cube intersecting $\Sigma_{i+1} + 4\gamma  \cdot L_u^{(i+1)}$ is good, and note that $4\gamma  > 1$. By Lemma~\ref{lem:construct-pinch-induction}, there exists an $i$-separated set $Z_i \subset \{u\}^\perp$ such that every $(i)$-cube intersecting $\Sigma_i + 4\gamma  \cdot L_u^{(i)}$ is good, as required.

It remains to prove that, for each $1 \leq i \leq k$, every $(i)$-cube intersecting $\Sigma + 3\gamma \cdot L_u^{(i)}$ is good. Since every $(i)$-cube intersecting $\Sigma_i + 4\gamma  \cdot L_u^{(i)}$ is good, it is enough to show that
\[
\Sigma + 3\gamma \cdot L_u^{(i)} \subset \Sigma_i + 4\gamma  \cdot L_u^{(i)}.
\]
To see that this holds, simply observe that
\[
\| h_1 - h_i \|_\infty \le 2^4 \gamma \sum_{j=1}^{i-1} \Delta_j \le 2^5 \gamma \cdot \Delta_{i-1} < \gamma \cdot \Delta_i,
\]
by Lemma~\ref{lem:pinched-height} and the triangle inequality, as required.
\end{proof}

\section{Construction of covers}\label{sec:covers}

To complete the proof of Proposition~\ref{prop:exists-T}, we shall show that one can cover each cluster of bad $(k)$-cubes by intersections of $(k,u)$-ranges with $u \in \str$, and observe that these intersections are $\U$-closed and well-separated from one another.

Let us fix the (arbitrary) set $A \subset \Z^d$ that appears in the statement of Proposition~\ref{prop:exists-T}, and set $\QQ_k := \QQ_k(A)$. Recall from Definition~\ref{def:Qk} that for each $Q \in \QQ_k$ we fix an element $x_Q \in Q$. We shall write $\edge T$ for the boundary of a set $T \subset \R^d$.

\begin{definition}\label{def:cover}
Let $k \ge 1$ and let $Q \in \QQ_k$. A \emph{$(k)$-cover} of~$Q$ is a set 
\[
T_k(Q) := \bigcap_{u \in \str} \Xi_u,
\]
where $\{\Xi_u : u \in \str\}$ is a set of $(k-1,u)$-ranges with bases
\begin{equation}\label{eq:cover-Piu}
\Pi_u := \big\{ x \in \R^d : \<x - x_Q, u\> = 3d \cdot \Delta_k \big\},
\end{equation}
such that 
\[
d\big( Q',\edge T_k(Q) \big) \geq 2\gamma \cdot \Delta_i
\]
for every $1 \leq i \leq k$ and every bad $(i)$-cube $Q'$, unless $i = k$ and $Q' \subset Q$.
\end{definition}

The first step is to use Lemma~\ref{lem:construct-pinch} to show that $(k)$-covers exist.

\begin{lemma}\label{lem:covers-exist}
For every $k \geq 1$ and $Q \in \QQ_k$, there exists a $(k)$-cover $T_k(Q)$ of $Q$ with 
\begin{equation}\label{eq:covers-small}
Q \subset T_k(Q) \subset B_{\gamma \Delta_k}(x_Q).
\end{equation}
\end{lemma}

\begin{proof}
Fix $k \ge 1$ and $Q \in \QQ_k$. For each $u \in \str$, we shall apply Lemma~\ref{lem:construct-pinch} to the hyperplane $\Pi_u$ defined in~\eqref{eq:cover-Piu}, with the set of infected sites being $A' := A \cap B_{2\gamma \Delta_k}(x_Q)$. We shall then take $T_k(Q)$ to be the intersection of the associated ranges.

To begin, we claim that every $(k)$-cube intersecting the set 
\[
\X_u := \big( \Pi_u + L_u^{(k)} \big) \cap B_{3\gamma \Delta_k}(x_Q)
\]
is good (with respect to the set $A$). Indeed, suppose that $Q_1$ is a bad $(k)$-cube that intersects $\X_u$. Then $Q_1 \not\subset Q$, since $Q$ does not intersect $\Pi_u + L_u^{(k)}$, by~\eqref{eq:buffer} and~\eqref{eq:cover-Piu}, and recalling that $x_Q \in Q$ and that $Q$ has diameter at most $2\sqrt{d} \cdot \Delta_k < (3d - 1) \cdot \Delta_k$. Now, by the maximality of $Q$ (see Definition~\ref{def:Qk}), it follows that there exists a bad $(k)$-cube $Q_2 \subset Q$ that is not adjacent to $Q_1$. Moreover, since $Q \in \QQ_k$, there exists a good $(k+1)$-cube $Q'$ that intersects $Q$. Observe that
$$\max \big\{ d(Q',Q_1), d(Q',Q_2) \big\} \le d(x_Q,Q_1) + 2\sqrt{d} \cdot \Delta_k \le \big( 3\gamma + 2\sqrt{d} \big) \Delta_k \le g_k,$$
since $Q_2 \subset Q$ and $\X_u \subset B_{3\gamma \Delta_k}(x_Q)$, and by~\eqref{def:Dk:gk}. Since $Q_1$ and $Q_2$ are non-adjacent bad $(k)$-cubes, it follows, by Definition~\ref{def:goodcubes}, that $Q'$ is bad, which is a contradiction. This contradiction proves that every $(k)$-cube intersecting $\X_u$ is good, as claimed.

Next we claim that every $(k)$-cube $Q'$ intersecting $\Pi_u + L_u^{(k)}$ is good with respect to the set $A'$. If $Q'$ intersects $\X_u$, then this follows from the claim above, since $A' \subset A$, so every cube that is good with respect to $A$ is also good with respect to $A'$. On the other hand, if $Q'$ does not intersect $\X_u$ then let $x \in Q' \cap \big( \Pi_u + L_u^{(k)} \big)$, and note that $\| x - x_Q \| > 3\gamma \Delta_k$. Since $Q'$ has diameter $2\sqrt{d} \cdot \Delta_k < \gamma \cdot \Delta_k$, it follows that $Q'$ does not intersect the ball $B_{2\gamma \Delta_k}(x_Q)$, and hence contains no point of $A'$. Therefore, in this case $Q'$ is automatically good with respect to $A'$, as claimed.

Applying Lemma~\ref{lem:construct-pinch} to $\Pi_u$ and $A'$, we obtain a $(k-1,u)$-pinch $\Sigma_u$, with base $\Pi_u$, such that for each $1 \leq i \leq k - 1$, every $(i)$-cube intersecting $\Sigma_u + 3\gamma \cdot L_u^{(i)}$ is good with respect to $A'$. We do this for each $u \in \str$, and define 
\[
T_k(Q) := \bigcap_{u \in \str} \Xi_u,
\]
where $\Xi_u$ is the $(k-1,u)$-range with boundary $\Sigma_u$. We shall prove, in the next two claims, that $T_k(Q)$ has the required properties.

\begin{claim}\label{clm:covers1}
$Q \subset T_k(Q) \subset B_{\gamma \Delta_k}(x_Q)$.
\end{claim}

\begin{clmproof}{clm:covers1}
It will be useful to consider the set $\tilde{T} := \bigcap_{u \in \str} \tilde{\Xi}_u$, where
\begin{equation}\label{def:Xi:tilde}
\tilde{\Xi}_u := \big\{ x \in \R^d : \< x,u \> \leq d \cdot \Delta_k \big\}
\end{equation}
for each $u \in \str$. We shall show that
\begin{equation}\label{eq:covers-nesting}
Q \subset x_Q + 3 \cdot \tilde{T} \subset T_k(Q) \subset x_Q + 4 \cdot \tilde{T} \subset B_{\gamma \Delta_k}(x_Q),
\end{equation}
which will imply the claim. Note first that $x_Q + 4 \cdot \tilde{T} \subset B_{\gamma \Delta_k}(x_Q)$ follows immediately from~\eqref{eq:gamma}. To prove the first three containments in~\eqref{eq:covers-nesting}, it is enough to show that
\begin{equation}\label{eq:covers-nesting2}
Q \subset x_Q + 3 \cdot \tilde{\Xi}_u \subset \Xi_u \subset x_Q + 4 \cdot \tilde{\Xi}_u
\end{equation}
for each $u \in \str$. The first containment in~\eqref{eq:covers-nesting2} holds because $x_Q \in Q$ and the diameter of $Q$ is at most $2\sqrt{d} \cdot \Delta_k$, and the second because $\Pi_u$ is the boundary of $x_Q + 3 \cdot \tilde{\Xi}_u$, by~\eqref{eq:cover-Piu} and~\eqref{def:Xi:tilde}, and since $\Sigma_u$ has base $\Pi_u$, and the height functions defined in~\eqref{def:h} are non-negative. Finally, to show that $\Xi_u \subset x_Q + 4 \cdot \tilde{\Xi}_u$, observe that 
$$\<x - x_Q, u\> \le 3d \cdot \Delta_k + 2^5 \gamma \cdot \Delta_{k-1} < 4d \cdot \Delta_k,$$ by~\eqref{eq:cover-Piu} and Lemma~\ref{lem:pinched-height}. This proves~\eqref{eq:covers-nesting2}, and hence proves the claim.
\end{clmproof}

It only remains to show that there are no bad cubes close to the boundary of $T_k(Q)$, except possibly those in $Q$. 

\begin{claim}\label{clm:covers2}
If\/ $1 \le i \le k$ and $Q_1$ is a bad $(i)$-cube with
$$d\big( Q_1,\edge T_k(Q) \big) < 2\gamma \cdot \Delta_i,$$
then $Q_1 \subset Q$ and $i = k$. 
\end{claim}

\begin{clmproof}{clm:covers2}
We shall deal separately with the cases $i < k$ and $i = k$. Beginning with the latter case, suppose that $Q_1$ is a bad $(k)$-cube with $Q_1 \not\subset Q$ and $d(Q_1, T_k(Q)) < 2\gamma \cdot \Delta_k$. Since $T_k(Q) \subset B_{\gamma \Delta_k}(x_Q)$, by Claim~\ref{clm:covers1}, it follows that $d(Q_1,x_Q) < 3\gamma \cdot \Delta_k$, and hence
$$d(Q_1,Q') \le d(Q_1,x_Q) + d(x_Q,Q') < 3\gamma \cdot \Delta_k + 2\sqrt{d} \cdot \Delta_k \le g_k$$
for any $(k+1)$-cube $Q'$ that intersects $Q$. Now, since $Q_1 \not\subset Q$ and by the maximality of $Q$, there exists a bad $(k)$-cube $Q_2 \subset Q$ that is not adjacent to $Q_1$. Noting that $d(Q_2,Q') \le 2\sqrt{d} \cdot \Delta_k \le g_k$, it follows that $Q'$ is bad. Thus, since $Q'$ was an arbitrary $(k+1)$-cube intersecting $Q$, this contradicts our assumption that $Q \in \QQ_k$.  

So suppose that $1 \le i \le k - 1$, let $Q_1$ be an $(i)$-cube with $d(Q_1, \edge T_k(Q)) < 2\gamma \cdot \Delta_i$, and note that therefore $d(Q_1, \Sigma_u) < 2\gamma \cdot \Delta_i$ for some $u \in \str$. We shall use the fact that every $(i)$-cube intersecting $\Sigma_u + 3\gamma \cdot L_u^{(i)}$ is good with respect to $A' = A \cap B_{2\gamma \Delta_k}(x_Q)$, which holds by our choice of $\Sigma_u$. The first step is to show that $Q_1$ intersects $\Sigma_u + 3\gamma \cdot L_u^{(i)}$. To do this, let $x \in Q_1$ and $y \in \Sigma_u$ with $d(x,y) < 2\gamma \cdot \Delta_i$, and let $x = x' + \lambda u$, where $x'\in \Sigma_u$ and $\lambda \in \R$. Observe that 
\[
|\lambda| = |\< x - x', u \>| \le d(x,y) + |\< x' - y, u \>|,
\]
and that, since $x',y \in \Sigma_u$,  
\[
|\< x' - y, u \>| \le 2^{10} \gamma \Delta_1^{1 - \beta} \cdot  d(x,y) \le d(x,y) / 2,
\]
by Lemma~\ref{lem:pinched-height}. Hence $|\lambda| \leq (3/2) \cdot d(x,y)$, and therefore, since $d(x,y) < 2\gamma \cdot \Delta_i$ and $x'\in \Sigma_u$, it follows that $x \in \Sigma_u + 3\gamma \cdot L_u^{(i)}$. Thus, by our choice of $\Sigma_u$, the $(i)$-cube $Q_1$ is good with respect to $A'$. 

To complete the proof, we shall show that $Q_1$ is also good with respect to $A$. To see that this holds, observe first that, as in the proof of Lemma~\ref{lem:k2ind}, the state of $Q_1$ depends only on the intersection of $A$ with the set of $x \in \Z^d$ such that
$$d(x,Q_1) \le \sum_{j = 1}^{i-1} \big( g_j + \sqrt{d} \cdot \Delta_j \big) \le 3 \cdot \Delta_{i-1}^\beta < \Delta_i / 3.$$
Since $d(Q_1, T_k(Q)) < 2\gamma \cdot \Delta_i$ and $T_k(Q) \subset B_{\gamma \Delta_k}(x_Q)$, by Claim~\ref{clm:covers1}, it follows that the state of $Q_1$ depends only on the set of $x \in \Z^d$ such that
$$d(x,x_Q) \le \gamma \cdot \Delta_k + 2\gamma \cdot \Delta_i + 2\sqrt{d} \cdot \Delta_i + \Delta_i / 3 \le 2\gamma \cdot \Delta_k.$$
Since $A' = A \cap B_{2\gamma \Delta_k}(x_Q)$, this proves the claim. 
\end{clmproof}

Combining Claims~\ref{clm:covers1} and~\ref{clm:covers2}, it follows that $T_k(Q)$ is a $(k)$-cover of $Q$, and that the inclusions~\eqref{eq:covers-small} hold, as required.
\end{proof}



Next we note that each individual $(k)$-cover is closed.

\begin{lemma}\label{lem:covers-props}
If\/ $T_k(Q)$ is a $(k)$-cover of\/ $Q \in \QQ_k$, then\/ $T_k(Q)_{\Z}$ is\/ $\U$-closed.
\end{lemma}

\begin{proof}
Recall that, by Lemma~\ref{lem:range-closed}, the set $\Xi_\Z$ is $\U$-closed for every $u \in \str$ and every $(k,u)$-range $\Xi$. Since $T_k(Q)$ is an intersection of $(k,u)$-ranges with $u \in \str$, it follows that $T_k(Q)_{\Z}$ is an intersection of $\U$-closed sets, and therefore is itself $\U$-closed, as required.
\end{proof}

We need one more simple lemma to complete the proof of Proposition~\ref{prop:exists-T}. Let us say that sets $\X,\Y \subset \R^d$ are \emph{strongly disjoint} if $d(\X,\Y) > 2R$, where $R := \max_{x \in X \in \U} \|x\|$. 

\begin{lemma}\label{lem:strongly:disjoint}
Let $1 \le i \le k$, and let $Q \in \QQ_i$ and $Q' \in \QQ_k$, with $Q \ne Q'$. If\/ $T_i(Q)$ is an $(i)$-cover of $Q$ and $T_k(Q')$ is a $(k)$-cover of $Q'$, with
\begin{equation}\label{eq:strongly:disjoint:condition}
T_i(Q) \subset B_{\gamma \Delta_i}(x_Q) \qquad \text{and} \qquad T_k(Q') \subset B_{\gamma \Delta_k}(x_{Q'}),
\end{equation}
then either $T_i(Q) \subset T_k(Q')$, or the sets $T_i(Q)$ and $T_k(Q')$ are strongly disjoint.
\end{lemma}

\begin{proof}
We consider the cases $i = k$ and $i < k$ separately. If $i = k$, then let $Q_1 \subset Q$ and $Q_2 \subset Q'$ be non-adjacent bad $(k)$-cubes, which exist by Definition~\ref{def:Qk}, since $Q \ne Q'$. Now, let $Q^*$ be a good $(k+1)$-cube intersecting $Q$. If $d(Q,Q') \le g_k/2$, then
$$\max\big\{ d(Q^*,Q_1), d(Q^*,Q_2) \big\} \le d(Q,Q') + 4\sqrt{d} \cdot \Delta_k \le g_k,$$
since $Q$ and $Q'$ each have diameter at most $2\sqrt{d} \cdot \Delta_k$. By Definition~\ref{def:goodcubes} this contradicts our assumption that $Q^*$ is good, and therefore $d(Q,Q') \ge g_k/2$. It follows, by~\eqref{def:Dk:gk} and~\eqref{eq:strongly:disjoint:condition}, that 
$$d\big( T_k(Q), T_k(Q') \big) \ge d(x_Q,x_Q') - 2\gamma \Delta_k \ge g_k/2 - 2\gamma \Delta_k > 2R,$$
and hence the sets $T_k(Q)$ and $T_k(Q')$ are strongly disjoint.

On the other hand, if $i < k$, then let $Q_1 \subset Q$ be the bad $(i)$-cube containing $x_Q$. Recall from Definition~\ref{def:cover} that, since $T_k(Q')$ is a $(k)$-cover of $Q'$, we have
$$d\big( x_Q,\edge T_k(Q') \big) \ge d\big( Q_1,\edge T_k(Q') \big) \ge 2\gamma \cdot \Delta_i.$$ 
Now, since $T_i(Q) \subset B_{\gamma \Delta_i}(x_Q)$, by~\eqref{eq:strongly:disjoint:condition}, it follows that 
$$d\big( T_i(Q), \edge T_k(Q') \big) \ge d\big( x_Q,\edge T_k(Q') \big) - \gamma \cdot \Delta_i \ge \gamma \cdot \Delta_i > 2R,$$
and hence either $T_i(Q) \subset T_k(Q')$, or the sets $T_k(Q)$ and $T_k(Q')$ are strongly disjoint, as required.
\end{proof}

We are finally ready to prove Proposition~\ref{prop:exists-T}. 

\begin{proof}[Proof of Proposition~\ref{prop:exists-T}]
For each $k \ge 1$ and each $Q \in \QQ_k(A)$, let $T_k(Q)$ be the $(k)$-cover of $Q$ given by Lemma~\ref{lem:covers-exist}, so $Q \subset T_k(Q) \subset B_{\gamma \Delta_k}(x_Q)$. By Lemma~\ref{lem:covers-props}, the set $T_k(Q)_{\Z}$ is $\U$-closed for each $Q \in \QQ_k$, and by Lemma~\ref{lem:strongly:disjoint}, for each $Q \in \QQ_i$ and $Q' \in \QQ_k$ the sets $T_i(Q)_\Z$ and $T_k(Q')_\Z$ are either nested or strongly disjoint. Defining $\T$ as in~\eqref{def:T}, it follows that $\T_\Z$ is $\U$-closed, as required.
\end{proof}

\section{The update families with $p_c(\Z^d,\U) = 1$}\label{sec:pc=1}

In this final section we shall show how to deduce the following theorem from a technical lemma that was proved in~\cite{BBMSupper}. 

\begin{theorem}\label{thm:pc=1}
Let $\U$ be a $d$-dimensional update family. Then
\[
p_c(\Z^d,\U) = 1 \qquad \Leftrightarrow \qquad \S(\U) = \SS^{d-1}.
\]
\end{theorem}

In order to avoid repetition, let us fix a $d$-dimensional update family $\U$ for the rest of the section. We begin with the easier of the two implications in the theorem, which is dealt with in the following lemma. 

\begin{lemma}\label{lem:pc=1-easy}
If\/ $\S(\U) \neq \SS^{d-1}$, then\/ $p_c(\Z^d,\U) < 1$.
\end{lemma}

\begin{proof}
Let $\Lambda$ be the graph on $\Z^d$ with edges between all pairs of sites at $\ell_\infty$ distance at most 1. It is easy to see by a standard argument that $q_c(d)$, the critical probability for percolation in $\Lambda$, is strictly positive. Indeed, if $X_n$ denotes the number of paths of open sites of length $n$ starting at the origin, where each site is open independently with probability $q$, then $\Ex[X_n] \le 3^{dn} q^n$. Hence, if $q$ is sufficiently small, then with probability $1$ there is no infinite component of open sites. Now recall that $R = \max_{x \in X \in \U} \|x\|$, and choose $p$ such that 
$$1 - q_c(d) < p^{(2R)^d} < 1.$$ 
We claim that $\P_p\big( [A]_\U = \Z^d \big) = 1$, and hence that $p_c(\Z^d,\U) \le p < 1$. 

To prove this, we tile $\Z^d$ with boxes of the form $\big(x + [0,2R)^d\big) \cap \Z^d$, and say that each box is `complete' if it is entirely contained in $A$, and `incomplete' otherwise.  By coupling with site percolation on $\Lambda$, we see that, with probability 1, every connected component of incomplete hypercubes is finite. Moreover, each site in such a connected component $C$ lies at distance at least $2R$ from any uninfected site in a different component. 

Now, let $u \in \SS^{d-1} \setminus \S(\U)$, and let $X \in \U$ be such that $X \subset \H_u$. Choose $y \in C$ with $\<y,u\>$ minimal, and observe that $y + X \subset A$, since $\<x,u\> < 0$ and $\|x\| \le R$ for every $x \in X$. Continuing in this way, we may infect (one by one) each of the sites in $C$, in increasing order of their inner product with $u$.
\end{proof}

The proof of the reverse implication hinges on the following deterministic lemma, which says that the $\U$-closure of the complement of a sufficiently large ball is not the whole of $\Z^d$. Recall that we defined $B_\lambda(x)$ in~\eqref{def:Br} to be the Euclidean ball of radius $\lambda$ centred at $x$.

\begin{lemma}\label{lem:pc=1-big-ball}
If\/ $\S(\U) = \SS^{d-1}$, then
\[
\big[ \Z^d \setminus B_\lambda(\0) \big]_\U \neq \Z^d
\]
for all sufficiently large $\lambda > 0$.
\end{lemma}

In order to prove this lemma, we shall use a construction from~\cite{BBMSupper} of a certain set $\QQ \subset  \SS^{d-1}$, which is called the set of `quasistable directions'. More precisely, we shall define a polytope
\[
P := \bigcap_{u \in \QQ} \big\{ x \in \R^d \,:\, \<x,u\> \leq 1 \big\},
\]
and show that $\lambda \cdot P$ cannot be invaded from outside in the $\U$-bootstrap process if $\lambda$ is sufficiently large. We state here only the properties of $\QQ$ that we need in order to prove Lemma~\ref{lem:pc=1-big-ball}, and refer the reader to Sections~3.3 and~6 of~\cite{BBMSupper} for further details. 

In order to state the two properties of $\QQ$ that we require, we need to define the following graph, which encodes which pairs of faces of $P$ are adjacent. 

\begin{definition}\label{def:voronoi}
Given a finite set $\QQ \subset \SS^{d-1}$ and $u \in \QQ$, the \emph{Voronoi cell} of $u$ with respect to $\QQ$ is 
$$\Cell_\QQ(u) := \big\{ w \in \SS^{d-1} : \<u,w\> \ge \<v,w\> \, \text{ for all } \, v \in \QQ \big\}.$$
The \emph{Voronoi graph} $\Vor(\QQ)$ has vertex set $\QQ$ and edge set 
$$E\big(\Vor(\QQ)\big) := \big\{ uv : \Cell_\QQ(u) \cap \Cell_\QQ(v) \ne \emptyset \big\}.$$
\end{definition}

It is not difficult to show (see~\cite[Section~8]{BBMSupper}) that if the face of $P$ corresponding to\footnote{The face of $P$ corresponding to $W$ is the set $\{ x \in P : \<x,u\> = 1 \text{ for each } u \in W \}$.} a set $W \subset \QQ$ is non-empty, then $W$ is a clique in $\Vor(\QQ)$.

Having defined the Voronoi graph, we can now state the following lemma from~\cite{BBMSupper}, which says that a suitable set of quasistable directions exists. The lemma is proved in~\cite[Section~6]{BBMSupper}; more precisely, it follows from~\cite[Lemmas~6.2 and~6.4]{BBMSupper}.

\begin{lemma}\label{lem:quasi-innerprod}
There exists a finite set $\QQ \subset \SS^{d-1}$, intersecting every open hemisphere of $\SS^{d-1}$, such that if $uv \in E\big( \Vor(\QQ) \big)$, then there does not exist $x \in X \in \U$ such that
\begin{equation}\label{eq:quasi-innerprod}
\< u,x \> < 0 \qquad \text{and} \qquad \< v,x \> > 0.
\end{equation}
\end{lemma}

We also need the following lemma, which is a particular case of~\cite[Lemma 9.8]{BBMSupper}. The proof, which is relatively straightforward, is given in~\cite[Appendix~B]{BBMSupper}.

\begin{lemma}\label{lem:close-to-many-faces}
There exists $\delta  = \delta(\QQ) > 0$ such that the following holds. Let $W \subset \QQ$, and suppose that there exists $x \in P$ with
$$\< x,u \> \ge 1 - \delta$$
for every $u \in W$. Then $W$ is a clique in $\Vor(\QQ)$.
\end{lemma}

We are now ready to prove Lemma~\ref{lem:pc=1-big-ball}.

\begin{proof}[Proof of Lemma~\ref{lem:pc=1-big-ball}]
In order to prove the lemma, it is enough to show that if $\lambda > 0$ is sufficiently large and $D := (\lambda \cdot P) \cap \Z^d$, then $\Z^d \setminus D$ is $\U$-closed. 

Suppose therefore that $\Z^d \setminus D$ is not $\U$-closed, and let $x \in D$ and $X \in \U$ be such that $(x + X) \cap D = \emptyset$. This implies that, for each $y \in X$, there exists $u \in \QQ$ such that $\<x + y, u\> > \lambda$. Let $W$ be the set of all such $u$; that is,
\[
W := \bigcup_{y \in X} \big\{ u \in \QQ \,:\, \< x + y, u \> > \lambda \big\}.
\]
Now, if $u \in W$, then $\<x,u\> \ge \lambda - R \ge (1 - \delta)\lambda$, since $\|y\| \leq R$ for every $y \in X$ and $\lambda$ was chosen sufficiently large. By Lemma~\ref{lem:close-to-many-faces}, it follows that $W$ is a clique in $\Vor(\QQ)$.

To complete the proof, we claim that
\begin{equation}\label{eq:pc=1-clm}
\< y, u^* \> > 0
\end{equation}
for all $y \in X$, where $u^* := \sum_{u \in W} u$. This will then imply that $X \subset \H_{-u^*}^d$, and hence that $-u^* \notin \S(\U)$, contradicting our assumption that $\S(\U) = \SS^{d-1}$. To prove~\eqref{eq:pc=1-clm}, fix $y \in X$, and recall that there exists $v \in W$ such that $\<x+y,v\> > \lambda$, and therefore $\<y,v\> > 0$, since $x \in D \subset \lambda \cdot P$. Since $W$ is a clique in $\Vor(\QQ)$, it follows by Lemma~\ref{lem:quasi-innerprod} that $\<y,u\> \ge 0$ for all $u \in W$. Since we also have $\<y,v\> > 0$, we obtain~\eqref{eq:pc=1-clm}, as required.
\end{proof}

We can now prove the following lemma which, together with Lemma~\ref{lem:pc=1-easy}, implies Theorem~\ref{thm:pc=1}. 

\pagebreak

\begin{lemma}\label{lem:pc=1-hard}
If\/ $\S(\U) = \SS^{d-1}$, then $p_c(\Z^d,\U) = 1$.
\end{lemma}

\begin{proof}
The lemma is an almost immediate consequence of Lemma~\ref{lem:pc=1-big-ball}. Indeed, if $p < 1$ and $A$ is a $p$-random subset of $\Z^d$, then with probability~1 the set $\Z^d \setminus A$ contains a translate of $B_\lambda(\0) \cap \Z^d$ for every $\lambda > 0$. By Lemma~\ref{lem:pc=1-big-ball}, it follows that $A$ almost surely fails to percolate, as required. 
\end{proof}

\section{Acknowledgements}

We are grateful to Ivailo Hartarsky and R\'eka Szab\'o for letting us know about their alternative proof of Theorem~\ref{thm:subcrit}. 

\bibliographystyle{amsplain}
\bibliography{bprefs}

\providecommand{\noopsort}[1]{}
\providecommand{\bysame}{\leavevmode\hbox to3em{\hrulefill}\thinspace}
\providecommand{\MR}{\relax\ifhmode\unskip\space\fi MR }
\providecommand{\MRhref}[2]{%
  \href{http://www.ams.org/mathscinet-getitem?mr=#1}{#2}
}
\providecommand{\href}[2]{#2}
\begin{thebibliography}{10}

\bibitem{BBMSupper}
P.~Balister, B.~Bollob\'as, R.~Morris, and P.~Smith, \emph{The critical length
  for growing a droplet}, preprint, arXiv:2203.13808.

\bibitem{BBMSlower}
\bysame, \emph{Universality for monotone cellular automata}, preprint,
  arXiv:2203.13806.

\bibitem{BBPS}
P.~Balister, B.~Bollob\'as, M.~Przykucki, and P.~Smith, \emph{Subcritical
  $\mathcal{U}$-bootstrap percolation models have non-trivial phase
  transitions}, Trans. Amer. Math. Soc. \textbf{368} (2016), 7385--7411.

\bibitem{BDMS}
B.~Bollob\'as, H.~Duminil-Copin, R.~Morris, and P.~Smith, \emph{Universality of
  two-dimensional critical cellular automata}, Proc. Lond. Math. Soc., to
  appear.

\bibitem{BSU}
B.~Bollob\'as, P.~Smith, and A.~Uzzell, \emph{Monotone cellular automata in a
  random environment}, Combin. Probab. Comput. \textbf{24} (2015), no.~4,
  687--722.

\bibitem{CLR}
J.~Chalupa, P.L. Leath, and G.R. Reich, \emph{Bootstrap percolation on a
  {B}ethe lattice}, J. Phys. C \textbf{12} (1979), no.~1, L31--L35.

\bibitem{vE}
A.~{\noopsort{Enter}}{van Enter}, \emph{Proof of {S}traley's argument for
  bootstrap percolation}, J. Stat. Phys. \textbf{48} (1987), 943--945.

\bibitem{Hart}
I.~Hartarsky, \emph{$\mathcal{U}$-bootstrap percolation: critical probability,
  exponential decay and applications}, Ann. Inst. H. Poincar\'e Probab.
  Statist. \textbf{57} (2021), no.~3, 1255--1280.

\bibitem{HS22}
I.~Hartarsky and R.~Szab\'o, \emph{Subcritical bootstrap percolation via {T}oom
  contours}, preprint, arXiv:2203.16366.

\bibitem{Sch1}
R.~Schonmann, \emph{On the behavior of some cellular automata related to
  bootstrap percolation}, Ann. Probab. \textbf{20} (1992), no.~1, 174--193.

\end{thebibliography}

\end{document}